\documentclass[11pt]{amsart} 
\usepackage[margin=1in]{geometry}
\usepackage[textwidth=1in]{todonotes}
\usepackage{enumitem}

\theoremstyle{plain} 
\newtheorem{theorem}{Theorem}[section]

\newtheorem{lemma}[theorem]{Lemma} 

\newtheorem{proposition}[theorem]{Proposition}

\newtheorem{obs}[theorem]{Observation} 
\theoremstyle{remark} 
 
\bibstyle{plain}

\renewcommand{\epsilon}{\varepsilon}

\newcommand{\cgF}{\mathcal{F}}

\newcommand{\cgR}{\mathcal{R}}

\newcommand{\cgW}{\mathcal{W}}
\newcommand{\reals}{\mathbb{R}} 
\newcommand{\Inc}{\operatorname{Inc}}
 
\newcommand{\Max}{\operatorname{Max}}
\newcommand{\Min}{\operatorname{Min}} 
\newcommand{\ldim}{\operatorname{ldim}}
\newcommand{\Bdim}{\operatorname{Bdim}} 
\newcommand{\fdim}{\operatorname{fdim}}
\newcommand{\fldim}{\operatorname{fldim}} 
\newcommand{\ple}{\operatorname{ple}} 
\newcommand{\LP}{\operatorname{LP}}
\newcommand{\TB}{\operatorname{TB}}
\newcommand{\bTB}{\operatorname{\bf TB}}
\newcommand{\MD}{\operatorname{MD}}
\newcommand{\FLD}{\operatorname{FLD}}
\newcommand{\MFLD}{\operatorname{MFLD}}
\newcommand{\bal}{\mathrm{bal}}

\newcommand{\bst}{\mathrm{bst}}
\newcommand{\sat}{\mathrm{sat}}

\begin{document}

\title[FRACTIONAL LOCAL DIMENSION]{Fractional Local Dimension}

\author[SMITH]{Heather C. Smith} 
\address{Mathematics \& Computer Science Department\\
  Davidson College\\
  Davidson, North Carolina 28035}
\email{hcsmith@davidson.edu}

\author[TROTTER]{William T. Trotter} 
\address{School of Mathematics\\
  Georgia Institute of Technology\\
  Atlanta, Georgia 30332}
\email{trotter@math.gatech.edu}

\date{October 10, 2020}

\subjclass[2010]{06A07, 05C35, 05D10}

\keywords{Dimension, local dimension, fractional dimension, 
  fractional local dimension}

\begin{abstract} 
The original notion of \textit{dimension} for posets was introduced by
Dushnik and Miller in 1941 and has been studied extensively in the 
literature. In 1992, Brightwell and Scheinerman developed the notion 
of \emph{fractional dimension} as the natural linear programming 
relaxation of the Dushnik-Miller concept.  In 2016, Ueckerdt 
introduced the concept of local dimension, and in just three years, several 
research papers studying this new parameter have been published.
In this paper, we introduce and study fractional local dimension.  
As suggested by the terminology, our parameter is a common generalization 
of fractional dimension and local dimension.

For a pair $(n,d)$ with $2\le d<n$, we consider the poset $P(1,d;n)$ consisting 
of all $1$-element and $d$-element subsets of $\{1,\dots,n\}$ partially 
ordered by inclusion. This poset has fractional dimension $d+1$, but for 
fixed $d\ge2$, its local dimension goes to
infinity with $n$. On the other hand, we show that as $n$ tends to infinity, the
fractional local dimension of $P(1,d;n)$ tends to a value $\FLD(d)$ which 
we will be able to determine \emph{exactly}.  For all $d\ge2$, $\FLD(d)$ is
strictly less than $d+1$, and for large $d$, $\FLD(d)\sim
d/(\log d-\log\log d-o(1))$. As an immediate
corollary, we show that if $P$ is a poset, and $d$ is the maximum degree 
of a vertex in the comparability graph of $P$, then the fractional 
local dimension of $P$, is at most $2+\FLD(d)$.  Our arguments use both 
discrete and continuous methods. 
\end{abstract}


\maketitle

\section{Introduction}\label{sec:intro}

A non-empty family $\cgR$ of linear extensions of
a poset $P$ is called a \textit{realizer} of
$P$ when $x\le y$ in $P$ if and only if $x\le y$ in $L$ for each $L\in\cgR$. In
their seminal paper, Dushnik and Miller~\cite{bib:DusMil} defined
the \textit{dimension} of $P$, denoted $\dim(P)$, as the least positive
integer $d$ for which $P$ has a realizer $\cgR$ with $|\cgR|=d$.

For an integer $n\ge2$, the \textit{standard example} $S_n$ is a height~$2$ poset
with minimal elements $A=\{a_1,a_2,\dots,a_n\}$ and maximal elements
$B=\{b_1,b_2,\dots,b_n\}$ such that $a_i<b_j$ in $S_n$ if and only if $i\neq j$.
As noted in~\cite{bib:DusMil}, $\dim(S_n)=n$ for all $n\ge2$. Dimension is clearly
a monotone parameter, i.e., if $Q$ is a subposet of $P$, then $\dim(Q)\le\dim(P)$.
Accordingly, a poset which contains a large standard example as a subposet has large dimension. On
the other hand, posets which do not contain the standard example $S_2$ are called
\textit{interval orders}, and it is well known that interval orders can have 
arbitrarily large dimension. 

In this paper, we introduce and study a new parameter for posets which we call 
\textit{fractional local dimension}.  To discuss our new parameter,
we need some additional notation and terminology.   A \textit{partial linear 
extension}, abbreviated $\ple$, of a  poset $P$ is a linear 
extension of a subposet of $P$.  We do not require the subposet to be 
proper, so a linear extension is also a $\ple$.  We let $\ple(P)$ denote 
the family of all $\ple$'s of $P$.

When $P$ is a poset, a function $w$ assigning to each $M\in\ple(P)$ 
a real-valued weight $w(M)$, with $0\le w(M)\le 1$, will be 
called a \textit{local weight function for $P$}. For each $u\in P$, we 
then define the \textit{local measure of $u$ for $w$}, denoted $\mu(u,w)$ 
by setting:
\[ 
  \mu(u,w)=\sum\{w(M):M\in\ple(P), u\in M\}. 
\] 
In turn, we define the
\textit{local measure of $P$ for $w$}, denoted $\mu(P,w)$ by setting:
 
\[ 
  \mu(P,w)=\max\{\mu(u,w):u\in P\}. 
\]

A local weight function $w$ for $P$ will be called a \textit{fractional local
realizer} of $P$ provided for each ordered pair $(u,v)$ of elements
(not necessarily distinct) of $P$, if $u\le v$ in $P$, we have
\[ 
  \sum\{w(M): M\in\ple(P), u\le v\text{ in }M\}\ge 1,
\] 
 and if $u$ is incomparable to $v$ in $P$, we have
\[ 
  \sum\{w(M): M\in\ple(P), u> v\text{ in }M\}\ge 1. 
\] 
The \textit{fractional local dimension of $P$}, denoted $\fldim(P)$, is then defined
by:
\[ 
  \fldim(P)=\min\{\mu(P,w):w\text{ is a fractional local realizer of $P$}\}. 
\]

Since $\fldim(P)$ is the solution to a linear programming problem posed with
coordinates which are integers, there is an optimal solution in which all
weights are rational numbers.

Readers may note that our parameter becomes \textit{fractional dimension}
if we only allow positive weights on linear extensions.  This concept was
introduced in 1992 by Brightwell and Scheinerman~\cite{bib:BriSch} and has
been studied by several groups of researchers in the intervening years
(see~\cite{bib:FelTro}, \cite{bib:TroWin} and~\cite{bib:BiHaPo}, for
example).

On the other hand, if we only allow integer values of $0$ or $1$ for weights,
then we obtain \textit{local dimension}. This is a relatively new concept
introduced by Ueckerdt ~\cite{bib:Ueck} in 2016, but in the past two years, this
concept is already the subject of several research papers (see~\cite{bib:BPSTT}, 
\cite{bib:TroWal} and~\cite{bib:BoGrTr}, for example).  This research
on local dimension has also sparked renewed interest in a related 
parameter, known as \textit{Boolean dimension} introduced in 1989 by  
Ne\v{s}et\v{r}il and Pudl\'ak~\cite{bib:NesPud} (see also~\cite{bib:GaNeTa}, 
\cite{bib:FeMeMi}, \cite{bib:MeMiTr}, \cite{bib:BPSTT} and~\cite{bib:TroWal}, 
for example).

In the discussion to follow, we will denote the fractional
dimension, local dimension and Boolean dimension of a poset $P$
by $\fdim(P)$, $\ldim(P)$ and $\Bdim(P)$, respectively.

We list below some basic properties of fractional local
dimension.   The first three are immediate consequences of the
definition.  The remaining items in the list are non-trivial.
However, in each case, a proof can be obtained via 
a relatively straightforward modification of an argument, or arguments, 
given elsewhere in the literature.  Accordingly, for each of these
statements, we give citations to the relevant papers.
Our focus in this paper is on results which do not have
analogues in the literature and seem to require proof 
techniques which are completely novel in the combinatorics
of posets.

\begin{enumerate}
\item For any poset $P$, $\fldim(P) \leq \fdim(P)$ and $\fldim(P)\leq \ldim(P)$.
\item If $P$ is a poset and $Q$ is a subposet of $P$, then
$\fldim(Q)\le\fldim(P)$.
\item If  $P$ is a poset and $P^*$ is the dual of $P$, then
$\fldim(P^*)=\fldim(P)$.
\item If $P$ is a poset, then the following statements are
equivalent:\\
(a)~$\dim(P)=2$; (b)~$\fdim(P)=2$; (c)~$\ldim(P)=2$; 
(d)~$\fldim(P)=2$ and (e)~$\Bdim(P)=2$.  The proof for fractional local dimension is
an easy extension of ideas from the proof for fractional dimension
given in~\cite{bib:BriSch} and the proof for local dimension is
given in~\cite{bib:BPSTT}.  
\item If $P$ is a poset with at least two points, then for every $x\in P$,
$\fldim(P)\le 1+\fldim(P-\{x\})$.  Again, the proof is a blending
of ideas from the proof of the analogous result for fractional dimension
(see either~\cite{bib:BriSch} or~\cite{bib:BiHaPo}) and the proof for
local dimension given in~\cite{bib:BPSTT}.  The question as to whether
the analogous result holds for Boolean dimension is still open.
\item For all $n\ge3$, $\ldim(S_n)=3$, as noted by Ueckerdt~\cite{bib:Ueck}. 
This observation was part of his motivation in developing the
concept of local dimension.  On the other hand, as noted
in~\cite{bib:BriSch}, $\fdim(S_n) = n$ for all $n\ge2$. It is a nice 
exercise to show that $\fldim(S_n) < 3$ for all $n\ge3$.  
\item If $P$ is an interval order, then it is shown in~\cite{bib:BriSch}
that $\fdim(P) < 4$.  In~\cite{bib:TroWin}, it is shown that this
inequality is asymptotically tight.  Both statements also hold
for fractional local dimension, and the argument for the tightness
of the inequality is a technical, but relatively 
straightforward, extension of the
proof in~\cite{bib:TroWin}.  On the other hand, it is shown
in~\cite{bib:BPSTT} that both local dimension and Boolean dimension
are unbounded on the class of interval orders.
\item The following conjecture, known as the \emph{Removable Pair Conjecture},
has been open for nearly 70 years: If $P$ is a poset with at least 
three points, then there is a distinct pair $\{x,y\}$ of points of $P$ such that 
$\dim(P)\le 1+\dim(P-\{x,y\})$.  In~\cite{bib:BiHaPo},
Bir\'{o}, Hamburger and P\'{o}r gave
a very clever proof of the analogous result for fractional dimension, and
it is straightforward to adapt their argument to obtain the parallel
result for fractional local dimension.  However,
the Removable Pair Conjecture is open for both
local dimension and Boolean dimension.
\end{enumerate}
 
\section{Forcing Large Fractional Local Dimension}\label{sec:force}

We have already noted that the fractional local dimension of a standard
example is less than~$3$ and the fractional local dimension of
an interval order is less than~$4$.  However,  there does not appear to
be an elementary argument to answer whether or not fractional local
dimension is bounded on the class of \emph{all} posets.
Accordingly, our primary goal for the remainder of this
paper is to analyze the fractional local dimension of posets in a well-studied
family. As a by-product of this work, we will learn that there are indeed posets
which have arbitrarily large fractional local dimension.

For a pair $(d,n)$ of integers 
with $2\le d < n$, let $P(1,d;n)$ denote
the height~$2$ poset consisting of all $1$-element and $d$-element subsets 
of $\{1,\dots,n\}$ ordered by inclusion. It is customary to ignore 
braces in discussing the minimal elements of $P(1,d;n)$ so that the 
minimal elements are just the integers in $\{1,\dots,n\}$.
Accordingly, $u<S$ in $P(1,d;n)$ when $u\in S$. We will abbreviate 
$\dim(P(1,d;n))$ as $\dim(1,d;n)$, with analogous abbreviations for 
other parameters.

Here are some of the highlights of results for posets in this family:

\begin{enumerate}
\item Dushnik~\cite{bib:Dush} calculated $\dim(1,d;n)$ exactly when 
$d\ge 2\sqrt{n}$.
\item  Combining the upper bound of Spencer~\cite{bib:Spen} with the 
lower bound of Kierstead~\cite{bib:Kier}, we have for fixed $d$:
\[ \dim(1,d;n) = \Omega\left(2^d\log\log n\right)\qquad \text{and}\qquad
 \dim(1,d;n)=O(d2^d\log\log n). \]

\item  For $d=2$, the value of $\dim(1,2;n)$ can be computed \textit{exactly}
for almost all values of $n$ (see the comments in~\cite{bib:BHPT}).
\item  It is shown in~\cite{bib:BPSTT} that for each fixed $d\ge2$, 
$\ldim(1,d;n)\rightarrow\infty$ with $n$.
\item  In~\cite{bib:BriSch}, it was noted that for each pair $(d,n)$
with $n>d\ge2$, $\fdim(1,d;n)=d+1$.
\end{enumerate}

For fixed $d\ge2$, $\fldim(1,d;n)$ is non-decreasing in $n$ and bounded from above
by $d+1$, so the following limit exists: 
\[
\FLD(d)=\lim_{n\rightarrow\infty}\fldim(1,d;n).
\]
Although we know $2<\FLD(d)\le d+1$ for all $d\ge2$, 
it is not immediately clear that $\FLD(d)<d+1$.
While this fact will emerge from our results, we know of no
simple proof.  Also, it is not clear whether $\FLD(d)$ grows
with $d$. Again, we will show that it does, but there does not appear to
be an elementary proof of this fact either.

For the remainder of this section, fix an arbitrary integer $d\geq 2$. 
To introduce our main results, we require some additional notation and
terminology.

Define a real-valued function $f(x)$ on the closed interval
$[0,1]$ by setting  $f(x)=(1-x)^d$.  Then there is a unique real number
$\beta$ with $0<\beta<1$ so that $f(\beta)=\beta$.  For example, when $d=2$,
\[
\beta = (3-\sqrt{5})/2\sim0.381966.
\]

We then consider real-valued functions $r(x)$ and $g(x)$, both
having the closed interval $[\beta,1]$ as their domain, defined by:
\begin{align}\label{eqn:r-and-g}
r(x)&=\frac{1}{d+1}\Bigl[(d+1)x - dx^{\frac{d+1}{d}} 
  - (1-x)^{d+1}\Bigr]\text{\quad and}\\
g(x)&=\frac{(d+1)r(x) - x}{(d+1)r(x) - x^2}.
\end{align}

Let $x_{\bst}$ and $x_{\bal}$ be the (uniquely determined)
numbers in $[\beta,1]$ such that (1)~the value of $g(x)$ is
maximized when $x=x_{\bst}$, and (2)~$r(x)=x^2/2$ when $x=x_{\bal}$. 
Simple calculations show for all $d\ge2$, $\beta< x_{\bst} < x_{\bal}<1$.
In turn, we define quantities $c_{\bal}$ and $c_{\bst}$ by:
\begin{align*}
c_{\bal} &= \frac{x}{r(x)}\text{\quad evaluated when $x=x_{\bal}$ and}\\
c_{\bst} &= 2 +\frac{(d-1)(x- x^2)}{(d+1)r(x) - x^2}\text{\quad evaluated when
$x=x_{\bst}$}.
\end{align*}

With these definitions in hand, we can now state our main theorem.

\begin{theorem}\label{thm:main}
For every $d\ge2$, $\FLD(d)=\min\{c_{\bal},c_{\bst}\}$.
\end{theorem}

Readers may note that our proof for Theorem~\ref{thm:main}
shows that $\FLD(d)$ is the minimum of the two quantities
$c_{\bst}$ and $c_{\bal}$ but does not tell us which
of the two is the correct answer.  However, with some
supplementary analysis, it is straightforward to determine
which of the two values is actually the minimum:

\begin{theorem}\label{thm:determination}
When $2\le d\le3$, $\FLD(d) =c_{\bal}$, and when $d\ge4$, $\FLD(d) = c_{\bst}$.
\end{theorem}

For a representative sampling of choices for $d$, the following
table gives the values of $x_{\bst}$, $x_{\bal}$, $c_{\bal}$ and
$c_{\bst}$, with all values truncated to six digits after the decimal 
point.

\begin{equation*}\label{tab:parameters}
\begin{array}{rrrrr}
d       & x_{\bst} &x_{\bal} &c_{\bal} & c_{\bst}\\
2       &0.790715 &0.800466 &2.498543     &2.499190\\
3       &0.667279 &0.675604 &2.960314     &2.960495\\
4       &0.560972 &0.589913 &3.390330     &3.389027\\
5       &0.478918 &0.526764 &3.796767     &3.790859\\
10      &0.289481 &0.356142 &5.615731     &5.560625\\
100     &0.048626 &0.071263 &28.065087    &26.812765\\
1000    &0.007055 &0.010988 &182.020489   &172.548703\\
10000   &0.000932 &0.001500 &1333.454911  &1269.860071\\
100000  &0.000116 &0.000191 &10458.564830 &10018.329933\\
1000000 &0.000014 &0.000023 &85721.802797 &82526.406188\\
\end{array}
\end{equation*}

Many extremal problems in combinatorics are relatively
straightforward if there is a single extremal example.  In our problem,
there are several natural candidates, and as suggested
by the statement of our main theorem, the optimum solution
is always to be found among two of these examples.

The following elementary proposition, proved
with standard techniques of calculus, shows the asymptotic behavior of $\FLD(d)$.

\begin{proposition}\label{pro:beta_d}
If $d\ge 4$, then $
  \FLD(d) = \frac{d}{\log d-\log\log d-o(1)}.$
\end{proposition}

Arguments for the supporting lemmas for our main theorems require us to prove
formally some highly intuitive statements.  Some of these statements 
involve properties of posets and $\ple$'s while others involve properties 
of real-valued functions, their derivatives, and intervals where their concavity is upwards and downwards. In cases where
elementary combinatorial arguments, algebraic manipulations and techniques from 
undergraduate analysis suffice, we will sometimes give only
the statement of lemmas and propositions.  In modestly more complex cases,
we will simply sketch the details of the proof.

The remainder of the paper is organized as follows.
In Section~\ref{sec:lower-bound}, we give the proof of the 
lower bound on $\FLD(d)$, i.e., we show that $\FLD(d)$ is at least the minimum 
of the quantities $c_{\bst}$ and $c_{\bal}$.  We consider this
lower bound our main result.   In Section~\ref{sec:upper-bound},
we give a construction showing that the lower bound is tight, a construction which capitalizes on 
the details of our proof for the lower bound. In particular, since
the lower bound is the minimum of two quantities, we develop
constructions which show that each of them is an upper bound on
$\FLD(d)$.

In Section~\ref{sec:maximum-degree}, we explain how our main theorem
implies that the fractional local dimension of  a poset $P$ is
at most $2+\FLD(d)$ when $P$ is a poset in which $d$ is the maximum degree
of its comparability graph.  Finally, in Section~\ref{sec:close}, we 
close with some brief comments on open problems for fractional local 
dimension. 

\section{Proof of the Lower Bound}\label{sec:lower-bound}

Fix an integer $d\ge2$.  In the first part of the proof,
we also fix an integer $n$ which is \emph{extremely} large
in comparison to $d$.  With the values of $d$ and $n$ fixed, we 
abbreviate $P(1,d;n)$ to $P$.

Throughout the proof, we will denote members of $\ple(P)$ with the
symbols $M$ and $L$, sometimes with subscripts or primes.  When
we use $M$, the $\ple$ may or may not be a linear extension, but
when we use $L$, the $\ple$ is \emph{required} to be a linear extension.

To provide readers with an overview\footnote{As is to be expected, this overview is
a slight distortion of the truth, but the errors disappear as
$n\rightarrow\infty$.} of where the argument is headed,
we will identify three $\ple$'s which will be called
$M_\bal$, $M_\bst$ and $M_\sat$.  Then we find that an optimal fractional
local realizer will assign positive weight \emph{only} to
$\ple$'s which are isomorphic to one of these three.  In fact,
either (1)~positive weight will be placed only on copies of $M_\bal$; or
(2)~positive weight will be placed only on copies of 
$M_\bst$ and $M_\sat$.

\subsection{The Relaxed Linear Programming Problem $\LP$}

Recall that finding the value of $\fldim(1,d;n)$ amounts to
solving a linear programming problem.
Taking advantage of the natural symmetry of the poset $P$, we formulate
a relaxed version of this problem using the following notation:

For each $M\in\ple(P)$,
\begin{enumerate}
\item $a(M)$ counts the number of elements of $\Min(P)$ which are
in $M$.
\item $b(M)$ counts the number of elements of $\Max(P)$ which are
in $M$.
\item $r(M)$ counts the number of pairs $(u,S)\in\Min(P)\times
\Max(P)$ with $u>S$ in $M$.
\item $q(M)$ counts the number of pairs $(u,S)\in\Min(P)\times\Max(P)$
with $u<S$ in $M$.
\end{enumerate}
We note that the count $q(M)$ is for all pairs
$(u,S)$ with $u<S$ in $M$, regardless of whether or not $u\in S$.
Therefore, we have the following basic identity which holds for
an arbitrary $\ple$ $M$: $r(M)+q(M) = a(M)\cdot b(M)$.

Here is the reduced problem which we will refer to as $\LP$:

\smallskip
\noindent
\textbf{LP}.\quad \textit{Minimize} the quantity $c$ for which there
is a local weight function $w:\ple(P)\rightarrow[0,1]$
satisfying the following constraints:
\begin{align*}\label{eqn:constraints-LP} 
  A(w)&=\sum_{M\in\ple(P)}a(M)\cdot w(M)\le c n.  \\
  B(w)&=\sum_{M\in\ple(P)}b(M)\cdot w(M)\le c \binom{n}{d}.
  \\
  R(w)&=\sum_{M\in\ple(P)}r(M)\cdot w(M)\ge n\binom{n-1}{d}.\\
  Q(w)&=\sum_{M\in\ple(P)}q(M)\cdot w(M)\ge n\binom{n}{d}. 
\end{align*} 

We refer to these four inequalities as
constraints~$A$, $B$, $R$ and~$Q$, respectively\footnote{Writing
the constraints in this in-line form serves to facilitate several arguments for
properties of solutions.  However, later in the proof,
we will consider them in the ``scaled'' form obtained by dividing both
sides of each constraint by the terms involving $n$ and/or $d$.}. The following observation relates the optimal solution for $\LP$ and the value of $\fldim(1,d;n)$.

\begin{obs}
If $w$ is a fractional local realizer of $P$ and
$c=\mu(P,w)$, then the pair $(c,w)$ is a solution for $\LP$. Therefore if $c_0$ is the minimum cost for $\LP$, then 
$\fldim(1,d;n)\ge c_0$.
\end{obs}
\begin{proof}
Observe that $A(w) = \sum_{a\in A} \mu(a,w)$. By the definition of $c$, and the fact that $|A(w)| = n$, we see $A(w) \leq cn$. Likewise, $B(w) = \sum_{b\in B} \mu(b,w) \leq \sum_{b\in B} c = c\binom{n}{d}$. In $P(1,d;n)$, each $a\in A$ is incomparable with $\binom{n-1}{d}$ elements of $B$. By the definition of a fractional local realizer, $\sum\{w(M): M\in\ple(P), u> v\text{ in }M\}\ge 1$ for each incomparable pair $(u,v)$. Therefore $R(w) \geq n\binom{n-1}{d}$.  Similarly, the definition of a fractional local realizer requires that every ordered pair $(u,v)$ of distinct elements with $u\leq v$ or $(v,u) \in \Inc P$ satisfies $\sum\{w(M): M\in\ple(P), u\le v\text{ in }M\}\ge 1$. Since there are $n\binom{n}{d}$ of these ordered pairs in $\Min P \times \Max P$ alone, $Q(w) \geq n \binom{n}{d}$. Therefore $(c,w)$ is a solutions for $\LP$. 
\end{proof}

Accordingly, to obtain a lower bound for 
$\fldim(1,d;n)$, we focus on finding a good
lower bound on $c_0$.  Although, the value of $c_0$ depends both
on $d$ and $n$, we will show that with $d$ fixed and
$n\rightarrow\infty$, $c_0$ tends to the minimum of $c_{\bst}$ and $c_{\bal}$.

\subsection{Properties of Optimum Solutions}

Let $\cgW_0$ denote the family of all local weight functions $w$ 
such that $(c_0,w)$ is an optimum solution for $\LP$.  In this
part of the argument, we develop four tie-breaking
rules: $\TB(1)$, $\TB(2)$, $\TB(3)$, and $\TB(4)$, which
will be used to  select a local weight function $w\in\cgW_0$
satisfying several useful properties.  Since the rules will not 
be presented at the same time, we use the ``computer science'' 
tradition for notation so that $\cgW_0$ consists of those 
optimum local weight functions which remain tied after all 
conditions developed to this point in the
argument have been applied.

The following elementary proposition is stated for emphasis.

\begin{proposition}\label{pro:two-tight}
If $w\in\cgW_0$, at least one of the two constraints
$A$ and~$B$ is tight, and at least one of the two constraints $R$ and $Q$
is tight.  
\end{proposition}

Here is our first tie-breaking rule for the selection of $w$:

\begin{enumerate}[wide=0pt]
\item[{$\bTB\bf{(1)}$}.] Maximize the quantity $R(w)+Q(w)-A(w)-B(w)$. 
\end{enumerate}

We note that $\TB(1)$ prefers optimum solutions where there
is a ``surplus'' in the constraints $R$ and $Q$ and/or ``slack''
in the constraints $A$ and $B$.  The following elementary
proposition follows from tie-breaker $\TB(1)$.

\begin{proposition}\label{pro:tb1-positive}
For $M\in \ple(P)$, if $w\in\cgW_0$ and $w(M)>0$, then both $a(M)$ and $b(M)$
are positive, so that at least one of $q(M)$ and $r(M)$
are positive.
\end{proposition}

\begin{proof}
For contradiction, suppose $a(M)=0$. (The case for $b(M)=0$ is similar.) Therefore $b(M)>0$ and $q(M) = r(M)= 0$. Create a new local weight function $w'$ which agrees with $w$ on all ple's except $w'(M) =0$. Observe $A(w') = A(w)$, $B(w') = B(w) - b(M)w(M)$, $Q(w') = Q(w)$, and $R(w') = R(w)$.  Therefore $w'$ is a solution to $\LP$ with cost $c' \leq c_0$. Since $w(M)$ and $b(M)$ are positive, $w'$ is preferred to $w$ by $\TB(1)$, a contradiction.  
\end{proof}

So for the remainder of the argument, we discuss only $\ple$'s
$M\in\ple(P)$ for which both $a(M)$ and $b(M)$ are positive.
This guarantees that at least one of $q(M)$ and
$r(M)$ is positive.  If $M$ is a $\ple$ of this type, there is a
uniquely determined integer $s =s(M)$, called the \textit{block-length}, with $0\le s\le n-d$ so that 
$M$ has the following block-structure:
\begin{equation}\label{eqn:block}
M=[A_0<B_1<A_1<B_2<A_2<\dots<B_s<A_s<B_{s+1}],
\end{equation}
where (1)~$\{B_1,B_2,\dots,B_{s+1}\}$ is a family of pairwise disjoint
subsets of $\Max(P)$ with $B_i\neq\emptyset$ when $1\le i\le s$, and
(2)~$\{A_0,A_1,\dots,A_s\}$ is a family of pairwise disjoint subsets
of $\Min(P)$ with $A_i\neq\emptyset$ when $1\le i\le s$. Note that we
allow either or both of $A_0$ and $B_{s+1}$ to be empty.  Also note that
$q(M)$ and $r(M)$ are determined entirely by the block-length $s$ of
$M$ and the sizes of the blocks in the block-structure of $M$.  Neither
value is affected by the arrangement of elements in a block.

For the balance of the argument, whenever we refer to a $\ple$ $M$
of $P$, we will use the letter $s=s(M)$ to denote the block-length
of $M$ and we will assume that the block-structure of $M$ is given
by~\eqref{eqn:block}.  Also, when we refer to a pair $(b,a)$, it
will always be the case that $a$ and $b$ are positive integers with
$1\le a\le n$ and $1\le b\le \binom{n}{d}$.

The statement of the following proposition, which again follows
immediately from $\TB(1)$, uses a convention for
subscripts that will be used for the balance of this section.
When we are discussing a finite family $\{M_1,M_2,\dots,M_t\}$
of  $\ple$'s, for each $i\in \{1,\dots,t\}$, we will
use the terms in $(a_i,b_i,s_i, r_i,q_i,w_i)$  as abbreviations for the terms in $(a(M_i),b(M_i),s(M_i), r(M_i),q(M_i),w(M_i)).$
If we make changes to a local weight function $w$ to form a local 
weight function $w'$, the value of $w'(M_i)$ is abbreviated as $w'_i$.

\begin{proposition}\label{pro:4-way}
Let $w\in\cgW_0$.  Then
there do not exist $\ple$'s $M_1$ and $M_2$ such that:
\begin{enumerate}
\item $w_1>0$,
\item $r_1/a_1\le r_2/a_2$, $r_1/b_1\le r_2/b_2$,
$q_1/a_1\le q_2/a_2$, and $q_1/b_1\le q_2/b_2$, and
\item at least one of the four inequalities in the preceding requirement is strict.
\end{enumerate}
\end{proposition}

\begin{proof}
For contradiction, suppose there exist $M_1$ and $M_2$ with the listed properties.  First assume both $r_2$ and $q_2$ are positive.  
By property (2), we can choose $\epsilon$ with $\max\{ r_1/r_2, q_1/q_2\} \leq \epsilon \leq \min\{ a_1/a_2, b_1/b_2\}$. By (3), the inequality between $\epsilon$ and one of these four fractions is strict. 

Define a new weight function $w'$ which agrees with $w$ on all ples except $w'_1 = 0$ and $w'_2 = w_2 + \epsilon w_1$. 
Now the following inequalities verify that $w'$ is a solution to $\LP$. 
\begin{eqnarray*}
A(w') &=& A(w) - w_1a_1 + \epsilon w_1a_2 \leq A(w) - w_1a_1 + \frac{a_1}{a_2} w_1a_2 = A(w).\\
B(w') &=& B(w) - w_1b_1 + \epsilon w_1b_2 \leq B(w) - w_1b_1 + \frac{b_1}{b_2} w_1b_2 = B(w)\\
R(w') &=& R(w) - w_1r_1 + \epsilon w_1r_2 \geq R(w) - w_1r_1 + \frac{r_1}{r_2} w_1r_2 = R(w)\\
Q(w') &=& Q(w) - w_1q_1 + \epsilon w_1q_2 \geq Q(w) - w_1q_1 + \frac{q_1}{q_2} w_1q_2 = Q(w)
\end{eqnarray*}
So $w'$ is a solution to $\LP$ and, by (3), one of these inequalities is strict. 
Therefore 
\[R(w') + Q(w') -A(w') -B(w') > R(w) + Q(w) -A(w) - B(w).\]
Thus $w'$ is preferred to $w$ by $\TB(1)$, a contradiction.

Now assume $r_2=0$. The argument for when $q_2=0$ is almost identical. By property (2),  $r_1=0$ and therefore $q_1$ and $q_2$ are positive by Proposition~\ref{pro:tb1-positive}. Since $r_1=r_2=0$, property (3) states that $q_1/a_1 < q_2/a_2$ or $q_1/b_1 < q_2/b_2$. Select $\epsilon$ with $q_1/q_2 \leq \epsilon \leq \min\{ a_1/a_2, b_1/b_2\}$, noting that the inequality between $\epsilon$ and one of these three fractions is strict. When we define $w'$ as above, we see $R(w')  = R(w)$, $A(w') \leq A(w)$, $B(w') \leq B(w)$, and $Q(w') \geq Q(w)$ where one of the three  inequalities is strict. As a result, $w'$ is preferred to $w$ by $\TB(1)$, a contradiction. 
\end{proof}

For a pair $(b,a)$, we let $\ple(b,a)$ denote the set
of all $\ple$'s $M$ with $a(M)=a$ and $b(M)=b$.
We say that a $\ple$ $M\in\ple(b,a)$ is \textit{maximal-preserving}
if there is no $M'\in\ple(b,a)$ with $q(M')>q(M)$.  
Determining the structure of maximal-preserving
$\ple$'s is trivial, since a $\ple$ $M$ is maximal-preserving
if and only if it has block-length~$0$ and block-structure
$M=[B_0<A_1]$.   Also, if $M\in\ple(b,a)$ and $M$ is maximal-preserving, 
then $r(M)=0$ and $q(M)=ab$.

\begin{lemma}\label{lem:q+1}
Let $M_1\in\ple(b,a)$.
If $M_1$ is not maximal-preserving, 
there exists a $\ple$ $M_2\in\ple(b,a)$ with $q_2= q_1+1$.
\end{lemma}

\begin{proof}
If $s_1=0$, then $M_1$ is maximal-preserving,
so we know that $s_1\ge1$.  Therefore, both $B_1$ and $A_1$
are non-empty and $r_1>0$.  Let $S$ be any set in $A_1$, and 
let $u$ be any element of $B_1$.
Form a $\ple$ $M_2$ by setting:
\[
M_2=[A_0<B_1-\{S\}<\{u\}<\{S\}<A_1-\{u\}<B_2<A_2<\dots<A_{s_1}<B_{s_1+1}].
\]
Clearly, $M_2$ satisfies the requirements of the lemma.
\end{proof}

\begin{lemma}\label{lem:r-tight}
If $w\in\cgW_0$, then constraint~$R$ is tight.
\end{lemma}

\begin{proof}
Let $w\in\cgW_0$.  We assume that constraint~$R$ is not tight for
$w$ and argue to a contradiction.  By Proposition~\ref{pro:two-tight}, constraint~$Q$ must be tight.  Thus there
is a positive number $\delta$ so that 
$R(w)=n\binom{n-1}{d}+\delta$ while $Q(w)=n\binom{n}{d}$.  

Let $M_1$ be any $\ple$ with $r_1$ and $w_1$ both positive.
Then $M_1$ is not maximal-preserving.  Let $M_2$ be the
$\ple$ provided by Lemma~\ref{lem:q+1}.

Let $\epsilon=\min\{w_1, \delta/2\}$.
Form a local weight function $w'$ by making the following
two changes to $w$: set $w'_1=w_1-\epsilon$ and 
$w'_2=w_2+\epsilon$.  Then $A(w')=A(w)$ and $B(w')=B(w)$.
Furthermore, $R(w')=R(w)-\epsilon$ and
$Q(w')=Q(w)+\epsilon$.  Since $\epsilon\le\delta/2$,
$(c_0,w')$ is a feasible solution for $\LP$ with neither constraint~$R$
nor constraint~$Q$ being tight. The contradiction completes the
proof of the lemma.
\end{proof}
 
Analogous with our treatment for maximal-preserving $\ple$'s,
we will say that a $\ple$ $M\in\ple(b,a)$ is \textit{maximal-reversing}
if there is no $\ple$ $M'\in\ple(b,a)$ with $r(M')>r(M)$.
We are now ready for the analogue of Lemma~\ref{lem:q+1}.

\begin{lemma}\label{lem:r+1}
Let $M\in\ple(b,a)$ and let $s$ be the block-length of $M$.
Then $M$ is maximal-reversing if and only if $M$ satisfies the
following five properties:

\smallskip
\noindent
Property 1.\quad
If $0\le i\le s$, $u\in A_i$ and $S\in B_{i+1}$,
then $u\in S$.

\smallskip
\noindent
Property 2.\quad
If $1\le i\le s$ and $B_{i+1}\neq\emptyset$, then 
$|A_i|=1$.

\smallskip
\noindent
Property 3.\quad
If $1\le i\le s$ and $B_{i+1}\neq\emptyset$, then
$B_i$ consists of all sets $S\in\Max(P)$ such that
(1)~$A_{i-1}\subset S$ and (2)~$S$ contains no elements
of $A_{i+1}\cup A_{i+2}\cup\dots\cup A_s$.  

\smallskip
\noindent
Property 4.\quad 
$|B_{s+1}|=\max\{0, b-\binom{n-1}{d}\}$.

\smallskip
\noindent
Property 5.\quad
$|A_0|=\max\{0, a-(n-d)\}$.

Furthermore, if $M_1\in\ple(b,a)$ is not maximal-reversing, then 
there is $M_2\in\ple(b,a)$ with $r_2=r_1+1$.
\end{lemma}

\begin{proof}
Before we present the proof, we pause to make two observations.
First, it is easy to see that a $\ple$ satisfying Properties~1, 2 and~3,
and one of Properties~4 and~5 satisfies all five properties.
We have elected to state the hypothesis in this form so that
the argument will be symmetric.

Second, given a pair $(b,a)$, if $M_1$ and $M_2$ belong to
$\ple(b,a)$ and both $M_1$ and $M_2$ satisfy all five properties,
then it is easy to see that $M_1$ and $M_2$ have the same
block-length and the same block-structure.  Therefore, $r_1=r_2$ and
$q_1=q_2$.  Accordingly, to complete the proof of the lemma, we
need only show that if $M_1\in\ple(b,a)$ violates one of the five 
properties in the hypothesis, then there is a $\ple$ $M_2\in
\ple(b,a)$ with $r_2=r_1+1$.  

Now suppose that $M_1$ does not satisfy Property~1.
Then there is some $i$ with $0\le i\le s_1$ for which
there is an element $u\in A_i$ and a set $S\in B_{i+1}$
such that $u\not\in S$.  In this case, we form a $\ple$
$M_2\in\ple(b,a)$ by setting:
\begin{multline*}
M_2=[A_0<B_1<A_1<B_2<\dots<A_i-\{u\}<\{S\}<\{u\}<B_{i+1}-\{S\}<\\
 <A_{i+1}<B_{i+2}<\dots<A_{s_1}<B_{s_1+1}].
\end{multline*}
Then $r_2=r_1+1$.  Accordingly, we may assume that
$M_1$  satisfies Property~1.

Now suppose that $M_1$ violates Property~2.  Then there is
some $i$ with $1\le i\le s_1$ such that $B_{i+1}\neq\emptyset$ and
$|A_i|\ge2$.  Let $u$ and $u'$ be distinct elements of $A_i$.  
Then $u,u'\in S$ for every set $S\in B_{i+1}$ by Property 1.  Fix a set $S\in
B_{i+1}$.  Since $i\ge 1$, we know $B_i\neq\emptyset$.  Let
$S'$ be any set in $B_i$, and let $v$ be any element of $S'-S$.
Then let $T$ be the set obtained from $S$ by removing $u'$ and
replacing it with $v$.  Clearly, the set $T$ does not
belong to $M_1$.  Form $M_2\in\ple(b,a)$ by setting:
\begin{multline*}
M_2=[A_0<B_1<A_1<B_2<\dots<A_i-\{u'\}<\{T\}<\{u'\}<B_{i+1}-\{S\}<\\
 <A_{i+1}<B_{i+2}<\dots<A_{s_1}<B_{s_1+1}].
\end{multline*}
Again, we note that $r_2=r_1+1$.  Accordingly, we may assume
that $M_1$ satisfies Properties~1 and~2.

Now suppose that $M_1$ violates Property~3.  Let $i$ be the
least integer with $1\le i\le s_1$ for which $B_{i+1}\neq
\emptyset$ and there is a set $T$ such that (1)~$A_{i-1}\subset 
T$; (2)~$T$ contains no element of $A_i\cup A_{i+1}\cup\dots\cup A_{s_1}$; 
and (3)~$T\not\in B_i$.  Clearly, the set $T$ is not in $M_1$.  Let
$S$ be any set in $A_{i+1}$.  Form $M_2\in\ple(b,a)$ by setting:
\begin{multline*}
M_2=[A_0<B_1<A_1<B_2<\dots<B_{i}\cup\{T\}<A_{i}<B_{i+1}-\{S\}<\\
 <A_{i+1}<B_{i+2}<\dots<A_{s_1}<B_{s_1+1}].
\end{multline*}

It follows that $M_1$ satisfies Properties~1, 2 and~3.  In view of
our observations at the start of the proof, we are left with the
case that $M_1$ violates both Properties~4 and~5.  So the inequalities
$|B_{s_1+1}|\ge\max\{0, b-\binom{n-1}{d}\}$ and $|A_0|\ge\max\{
0, a-(n-d)\}$ must both be strict.  Now
let $u\in A_0$ and $S\in B_1$.  It follows that there
is a set $T\in\Max(P)$ such that $T$ contains no element
of $\{u\}\cup A_1\cup\dots\cup A_{s_1}$.  Then $T$ is not in
$M_1$.  Form a $\ple$ $M_2$ by setting:
\begin{multline*}
M_2=[A_0-\{u\}<\{T\}<\{u\}< B_1-\{S\}<A_1<B_2<A_2<\dots<
 B_{s_1}<A_{s_1}<B_{s_1+1}].
\end{multline*}
Since $r_2=r_1+1$, the contradiction completes the proof of the lemma.
\end{proof}

For the balance of the proof, given a pair $(b,a)$, we let
$M(b,a)$ denote a maximal-reversing $\ple$ from $\ple(b,a)$, and
we abbreviate $r(M(b,a))$ and $q(M(b,a))$ as $r(b,a)$ and
$q(b,a)$, respectively.

Now for the second tie-breaking rule:

\begin{enumerate}[wide=0pt]
\item[{$\bTB\bf{(2)}$}.] Minimize the number $t_2$ of $\ple$'s in $\ple(P)$
which are assigned positive weight by $w$ and are
neither maximal-preserving nor maximal-reversing.
\end{enumerate} 

\begin{lemma}\label{lem:decisive}
The value of $t_2$ is zero, i.e., there is a local weight
function $w\in\cgW_0$ such that if $w(M)>0$, then either
$M$ is maximal-preserving or maximal-reversing.
\end{lemma}

\begin{proof}
We argue by contradiction. Let $w\in\cgW_0$.  Then
there is a $\ple$ $M_1$ with $w_1>0$ such that
$M_1$ is neither maximal-preserving nor maximal-reversing.
Since $M_1$ is not maximal-preserving, $r_1>0$.  If $q_1=0$,
the block-length of $M_1$ is~$1$ and the 
block-structure of $M_1$ is $[B_1<A_1]$ which implies $M_1$ is 
maximal-reversing.  It follows that $q_1>0$.

Applying Lemma~\ref{lem:q+1} repeatedly to $M_1$, let $i_2$ be the largest
integer for which there is a $\ple$ $M_2$ such that
$a_2=a_1$, $b_2=b_1$, $q_2=q_1+i_2$ (and $r_2=r_1-i_2$).
Note that $i_2$ is positive and $M_2$ is maximal-preserving.
 
Applying Lemma~\ref{lem:r+1} repeatedly to $M_1$, let $i_3$ be the largest
integer for which there is a $\ple$ $M_3$ such that
$a_3=a_1$, $b_3=b_1$, $r_3=r_1+i_3$ (and $q_3=q_1-i_3$).
Again, we note that $i_3>0$ and $M_3$ is maximal-reversing.

We determine a local weight function $w'$ by making the following
three changes to $w$:  Set $w'_1=0$, $w'_2=w_2+w_1i_3/(i_2+i_3)$ 
and $w'_3=w_3+w_1i_2/(i_2+i_3)$.  With these values
$A(w')=A(w)$, $B(w')=B(w)$, $Q(w')=Q(w)$ and $R(w')=R(w)$, so
that $w'\in\cgW_0$.  However,
the value of $t_2$ has gone down by one.  This observation completes
the proof.
\end{proof}

Here is an another easy consequence of $\TB(1)$ and $\TB(2)$.

\begin{lemma}\label{lem:completely-regular}
Let $w\in\cgW_0$ and let $M$ be a $\ple$ with $w(M)>0$.
Then the following statements hold:
\begin{enumerate}
\item If $M$ is maximal-preserving, then $a(M)=n$ and 
$b(M)=\binom{n}{d}$ so that $M=L=[\Min(P)<\Max(P)]$ is a linear extension of
$P$.
\item If $M$ is maximal-reversing, $s=s(M)$ and
$k=|A_1|+|A_2|+\dots+|A_s|$, then $|B_1|=\binom{n-k}{d}$.
\end{enumerate}
\end{lemma}

\begin{proof}
To prove the first statement, we
assume that $M_1$ satisfies $w_1>0$, $r_1=0$ and
$a_1+b_1<n+\binom{n}{d}$.  We note that $q_1=a_1b_1$ and
$r_1=0$.  Therefore (1)~$q_1/a_1=b_1\le \binom{n}{d}$ and
(2)~$q_1/b_1=a_1\le n$.  Then let $M_2$ be a maximal-preserving
linear extension of $P$ with block-structure $M_2=[\Min(P)<\Max(P)]$.
It follows that the pair $(M_1,M_2)$ violates Proposition~\ref{pro:4-way}.

For the second statement, let $M_1$ be a maximal-reversing
$\ple$ with $w_1>0$, and let $k=|A_1|+|A_2|+\dots+|A_{s_1}|$.
If $a_1\ge n-d$, then $k=n-d$ and $|B_1|=1$, as required.  Now 
suppose that $a_1<n-d$.  Then $A_0=\emptyset$.  If
$s_1\ge2$, then $|B_1|=\binom{n-k}{d}$ follows
from Property~3.  

Now suppose $s_1=1$.  Then
$M_1$ has block-structure $[B_1<A_1]$.  Therefore $r_1=a_1b_1$ 
and $q_1=0$.
Suppose that $|B_1|<\binom{n-k}{d}=\binom{n-a_1}{d}$.
Let $M_2$ be a $\ple$ obtained from $M_1$ simply by adding to
$B_1$ a set $S$ which does contain any of the minimal elements
in $A_1$ and does not belong to $B_1$.  Then
$a_2=a_1$, $b_2=b_1+1$, $r_2=r_1+b_1$ and $q_2=q_1=0$. It
follows that the pair $(M_1,M_2)$ violates Proposition~\ref{pro:4-way}. 
\end{proof}

We will say that a maximal-reversing $\ple$ $M$ of block-length~$s$ 
is \textit{block-regular} if the following two conditions are met:

\begin{enumerate}
\item Either $A_0=\emptyset$ or $a=n$.
\item Either $b=\binom{n}{d}$ or $b=\binom{n-|A_s|}{d}$.
\end{enumerate}

Here is our third tie-breaking rule:

\begin{enumerate}[wide=0pt]
\item[{$\bTB\bf{(3)}$}.] Minimize the number $t_3$ of $\ple$'s in $\ple(P)$
which are assigned positive weight by $w$ and are maximal-reversing but 
not block-regular.
\end{enumerate} 

\begin{lemma}\label{lem:block-regular}
The value of $t_3$ is $0$, i.e., there is an optimum local weight function $w\in\cgW_0$
so that if $M\in\ple(P)$ and $w(M)>0$, then either $M$ is a 
maximal-preserving linear extension of $P$, or $M$ is block-regular.
\end{lemma}

\begin{proof}
We argue by contradiction and assume that $t_3>0$.  Let $w\in\cgW_0$ and
consider the non-empty family $\cgF$  of all $\ple$'s assigned 
positive weight which are maximal-reversing but not block-regular.  We show 
that we can devise a sequence of weight shifts to decrease the value
of $t_3$. Describing these shifts involves three cases, and we
give full details for one of them.
First, suppose that $M_1\in\cgF$, $1\le a_1\le n-d$ and $1\le b_1 < \binom{n-1}{d}$. In this case, we must have $0<|B_{s_1}|
<\binom{n-|A_{s_1}|}{d-1}$.  Note that this requires $s_1\ge2$.

Let $M_2$ be a block-regular $\ple$ with $a_2=a_1$ and
$s_2=s_1-1$.  Also, let $M_3$ be a block-regular $\ple$ with $a_3=a_1$ and
$s_3=s_1$.  Note that $b_2<b_1<b_3$.

Then there are positive numbers $\lambda_2$ and $\lambda_3$ with
$\lambda_2+\lambda_3=1$ so that $b_1=\lambda_2 b_2+\lambda_3b_3$.
Simple calculations then show that $r_1=\lambda_2 r_2+\lambda_3 r_3$ and
$q_1=\lambda_2q_2+\lambda_3 q_3$.  Form a local weight function
$w'$ from $w$ by making the following three changes:
$w'_1=0$, $w'_2=w_2+\lambda_2 w_1$ and $w'_3=w_3+\lambda_3 w_1$.
Then $A(w')=A(w)$, $B(w')=B(w)$, $R(w')=R(w)$ and $Q(w')=Q(w)$, so
$w'\in\cgW_0$.  However, the value of $t_3$ has gone down by~$1$.
The contradiction shows that if $M_1\in\cgF$, then either $n-d<a_1<n$ or
$\binom{n-1}{d}<b_1<\binom{n}{d}$. 

If $\binom{n-1}{d}<b_1<\binom{n-1}{d}$, then we shift weight from
$M_1$ to $M_2=M(\binom{n-1}{d},a_1)$ and $M_3=M(\binom{n}{d},a_1)$.  
If $n-d<a_1<n$, we shift weight from $M_1$ to $M_2=M(b_1,n-d)$ 
and $M_3=M(b_1,n)$.  After at most two applications of such
shifts, the value of $t_3$ has been decreased.  
\end{proof}

Considering the structure of $\LP$, the following terminology
is natural.  Let $M\in\ple(b,a)$  We will say that:

\begin{enumerate}
\item $M$ is \textit{balanced} if $bn = a\binom{n}{d}$.
\item $M$ is $A$-\textit{heavy} if $bn < a\binom{n}{d}$.
\item $M$ is $B$-\textit{heavy} if $bn> a\binom{n}{d}$.
\end{enumerate}
We note that linear extensions of $P$ are balanced, and so is the
block-regular $\ple$ $M(\binom{n-1}{d},n-d)$.
Due to the limitations of integer arithmetic,
we may expect that many maximal-reversing $\ple$'s are not balanced.
On the other hand, the following lemma asserts that collectively,
the family of $\ple$'s assigned positive weight is balanced.

\begin{lemma}\label{lem:A-and-B-tight}
If $w\in \cgW_0$, then both constraints $A$ and $B$ are tight.
\end{lemma}

\begin{proof}
We argue by contradiction and show that constraint~$B$ must be tight.  
The argument for constraint~$A$ is symmetric.  Since
constraint~$B$ is not tight, constraint~$A$ must be tight by Proposition~\ref{pro:two-tight}.  Accordingly,
there is a positive real number $\delta$ so that $A(w)=nc_0$ and
$B(w)=\binom{n}{d}c_0-\delta$.  In turn, this implies that
$w$ assigns positive weight to at least one $\ple$ which
is block-regular and $A$-heavy. 

Now suppose that $w$ assigns positive weight to a 
block-regular $\ple$ $M_1=M(b,a)$ with
$a_1<\binom{n-1}{d}$.  Let $M_2$ be the 
maximal-reversing (but not block-regular) $\ple$ $M(b_1+1,a_1)$.
We note that $r_2>r_1$ and $q_2>q_1$.

Let $\epsilon_1=\min\{w_1,\delta/2\}$.  Then set
$\epsilon_2=\max\{\epsilon_1r_1/r_2, \epsilon_1q_1/q_2\}$.
Then make the following two changes to $w$:  Set
$w'_1=w_1-\epsilon_1$ and $w'_2=w_2+\epsilon_2$.
Note first that $R(w')=R(w)-\epsilon_1 r_1+\epsilon_2r_2\ge R(w)$.
Similarly, $Q(w')\ge Q(w)$.  On the other hand, since $\epsilon_2<
\epsilon_1$, we have:
\[
A(w')=A(w)-\epsilon_1 a_1+\epsilon_2a_2 
=A(w)-\epsilon_1 a_1+\epsilon_2a_1
<A(w)
=nc_0.\]
Also, we have: 
\begin{align*}
B(w')&=B(w)-\epsilon_1 b_1+\epsilon_2 b_2\\
     &=B(w)-\epsilon_1 b_1+\epsilon_2 (b_1+1)\\
     &\le B(w)-\epsilon_1 b_1+\epsilon_1(b_1+1)\\
     &=B(w)+\epsilon_1\\
     &\le B(w)+\delta/2\\
     &<\binom{n}{d}c_0.
\end{align*}
It follows that $w'\in\cgW_0$ but neither constraint $A$ nor
constraint~$B$ is tight, contradicting Proposition~\ref{pro:two-tight}.  This forces 
$b(M)\ge\binom{n-1}{d}$ for every block-regular $\ple$ $M$ 
assigned positive weight by $w$.

  It follows that
$w$ assigns positive weight to the block-regular $\ple$
$M_1=M(\binom{n-1}{d},n)$.
Let $M_2=M(\binom{n-1}{d},n-d)$  and $M_3=M(\binom{n}{d},n)$.
We note that $r_1=r_2=r_3$ and $q_2<q_1<q_3$.

Let $\epsilon_1=\min\{w_1,\delta/(2\binom{n}{d})\}$.  Then
set $\epsilon_2=\epsilon_1(q_3-q_1)/(q_3-q_2)$ and
$\epsilon_3=\epsilon_1(q_1-q_2)/(q_3-q_2)$.  Form
$w'$ by making the following three changes. Set
$w'_1=w_1-\epsilon_1$, $w'_2=w_2+\epsilon_2$ and
$w'_3=w_3+\epsilon_3$.

Since $\epsilon_1=\epsilon_2+\epsilon_3$, we have
$R(w')=R(w)$.  An easy calculation shows that we
also have $Q(w')=Q(w)$.  For constraint~$A$, we have:
\[A(w')=A(w)-\epsilon_1 n+\epsilon_2(n-d)+\epsilon_3 n
<A(w)-\epsilon_1 n+\epsilon_2 n+\epsilon_3 n
=A(w)
= nc_0.\]
For constraint~$B$, we have:
\begin{align*}
B(w')&=B(w)-\epsilon_1\binom{n-1}{d}+\epsilon_2\binom{n-1}{d}+\epsilon_3\binom{n}{d}\\
     &=B(w)+\epsilon_3\left(\binom{n}{d}-\binom{n-1}{d}\right)\\
     &<B(w)+\epsilon_1\binom{n}{d}\\
     &\le B(w)+\delta/2\\
     &<\binom{n}{d}c_0.
\end{align*}
Again, these calculations show that $w'\in\cgW_0$ but neither constraint~$A$
nor constraint~$B$ is tight.  The contradiction completes the proof.
\end{proof}

We close this subsection with a two useful observations about
maximal preserving $\ple$'s.  The elementary arguments
are left as exercises.

\begin{proposition}\label{pro:max-q/r}
Fix an integer $a$ with $1\le a\le n$.  Then
the ratio $q(b,a)/r(b,a)$ is strictly increasing
in $a$.  Furthermore, over all pairs $(b,a)$, the
maximum value of $q(b,a)/r(b,a)$ is $d$ and
this occurs if and only if $a=n$ and $b=\binom{n}{d}$.
\end{proposition}

\subsection{Transitioning to the Continuous Setting}

In the next part of the discussion, we will keep the
value of $d$ fixed, but we will transition from the
discrete to continuous settings by analyzing the behavior of
solutions as $n\rightarrow\infty$. Since our emphasis is on 
combinatorics and not analysis, we proceed in an informal manner,
but it should be clear to the reader how our arguments can be 
recast in a completely formal setting.    

Fix a block-regular $\ple$ $M=M(b,a)$ with $1\le a\le n-d$ and $b\le
\binom{n-1}{d}$.
With $M$ fixed for the moment, we abbreviate the quantities in
$(a(M), b(M), r(M),q(M),s(M))$  as $(a,b,r,q,s)$, respectively.
We also set $m=|A_s|$, so that $b=\binom{n-m}{d}$.

We then have the following \textit{exact} formula for $r=r(M)$:
\begin{align}\label{eqn:exact}
\begin{split}
 r&=\binom{n-a}{d}+\binom{n-a+1}{d}+\dots+\binom{n-m-1}{d}+m\binom{n-m}{d} \\
  &=m\binom{n-m}{d}+\binom{n-m}{d+1}-\binom{n-a}{d+1}\\
  &=\frac{1}{d+1}\Bigl[(n+md-d)\binom{n-m}{d}-(n-a-d)\binom{n-a}{d}\Bigr]
  \end{split}.
\end{align}

We define quantities $x$ and $y$ by setting
$x=b/\binom{n}{d}$ and $y = a/n$.  Since $b=\binom{n-m}{d}$, it
follows that (in the limit) $x=(1-m/n)^d$.  Since $a\ge m$, it 
follows that we have the following restrictions on the pair $(x,y)$:
\begin{equation}\label{eqn:x-y}
 0\le x,y\le 1  \qquad \qquad y \ge 1-x^{1/d}.
\end{equation}

Previously, we analyzed the quantity $r(b,a)$.  Now we study
the ratio $r(x,y)$ which we define by setting:
\[
r(x,y)=\frac{r(b,a)}{n\binom{n}{d}}.
\]

We then have the following formula for $r(x,y)$:
\begin{equation}\label{eqn:r(x,y)}
r(x,y)=\frac{1}{d+1}\Bigl[(d+1)x - dx^{\frac{d+1}{d}} 
  - (1-y)^{d+1}\Bigr].
\end{equation}

It is not necessary to develop a parallel formula for $q(x,y)$ since
$q(x,y)=xy-r(x,y)$.  In the continuous setting, the linear
programming problem $\LP$ becomes:

\smallskip
Minimize the quantity $c$ subject to the following constraints:

\begin{align*}\label{eqn:constraints-LP} 
  A(w)&=\sum_{M\in\ple(P)}y(M)\cdot w(M)\le c.\\  
  B(w)&=\sum_{M\in\ple(P)}x(M)\cdot w(M)\le c.\\ 
  R(w)&=\sum_{M\in\ple(P)}r(M)\cdot w(M)\ge 1.\\
  Q(w)&=\sum_{M\in\ple(P)}q(M)\cdot w(M)\ge 1. 
\end{align*} 

We continue to refer to these inequalities as
constraints~$A$, $B$, $R$ and~$Q$, respectively. 
 
\subsection{Ratio Curves}

Let $\rho$ be a non-negative number.  We call the set
of all pairs $(x,y)$ such that $q(x,y)/r(x,y)=\rho$ the 
$\rho$-\textit{ratio curve}.  One special case is $\rho=0$ and we 
note that the $0$-ratio curve is just the set of pairs $(x,y)$ with
$y=1-x^{1/d}$.

\begin{lemma}\label{lem:ratio-curve}
For each $\rho\in\reals_0$, the pairs $(x,y)$ on the $\rho$-ratio
curve determine a decreasing function of $x$.
\end{lemma}

\begin{proof}
The equation $q(x,y)/r(x,y)=\rho$ is equivalent to $xy/r(x,y)=1+\rho$.
Using implicit differentiation and abbreviating $r(x,y)$ as
$r$, we obtain the following formula for the derivative $\frac{dy}{dx}$ at the
point $(x,y)$:
\[
\frac{dy}{dx}=-\frac{y[r-x(1-x^{1/d})]}{x[r-y(1-y)^d]}.
\]
In this quotient, when $0<x,y<1$, the numerator 
is positive since $r>x(1-x^{1/d})$.  Note that, in the discrete
setting, this is just saying that $r>bm$.  Similarly, the denominator
is positive since $r>y(1-y)^d$.  In the discrete setting, if $b_1$
is the size of the bottom block of maximal elements, this inequality
is simply the statement that $r>ab_1$.  Therefore $\frac{dy}{dx}<0$ and
the pairs on the $\rho$-ratio curve form a decreasing function in $x$. 
\end{proof}

\begin{lemma}
Let $\rho$ be a fixed non-negative number and consider the set
of all pairs $(x,y)$ on the $\rho$-ratio curve with $x,y>0$.  Among these
pairs, the quantity $r(x,y)$ is maximized and the
quantity $(x+y)/r(x,y)$ is minimized when $x=y$.
\end{lemma}

\begin{proof}
We outline the proof for the first statement.  The argument for the second
is similar.  On the $\rho$-ratio curve, $y$ is a function of $x$, and
we have a formula for $\frac{dy}{dx}$ 
from the proof of the preceding lemma.  It follows that
on this ratio curve, $r(x,y)$ is a function of $x$.  Again, using 
implicit differentiation, we have:
\[
\frac{dr}{dx}=  1-x^{1/d}+(1-y)^d \frac{dy}{dx}.
\]
From this expression, it is straightforward to verify that
$\frac{dr}{dx}$ is positive when $x<y$, $0$ when $x=y$ and negative when 
$x>y$.  Therefore, on the $\rho$-ratio curve, $r(x,y)$ is maximum
when $x=y$.
\end{proof}

\begin{lemma}\label{lem:balanced}
If $M=M(x,y)$, $w\in\cgW_0$ and $w(M)>0$, then $M$ is balanced.
\end{lemma}

\begin{proof}
We argue by contradiction and suppose that $w$ assigns positive weight
to a $\ple$ which is not balanced. 
Since constraints $A$ and $B$ are tight, it follows that
there are $\ple$'s $M_1=N(x_1,y_1)$ and $M_2=N(x_1,y_2)$, both
with positive weight, so that $M_1$ is $B$-heavy and $M_2$ is $A$-heavy.

For $i\in \{1,2\}$, let $\rho_i=q_i/r_i$.  Let $M_3=M(x_3,y_3)$ 
be the point on the $\rho_1$-ratio curve where $x_3=y_3$.  From the
preceding lemma, we know that $r_3>r_1$. Also, we know
that $2x_3/r_3<(x_1+y_1)/r_1$.  Also, let $M_4=N(x_4,y_4)$ 
be the point on the $\rho_2$-ratio curve where $x_4=y_4$.  From the
preceding lemma, we know that $r_4>r_2$. Also, we know
that $2x_4/r_4<(x_2+y_2)/r_2$.

Let $w'$ be the local weight function obtained from $w$ by making
the following  changes: $w'_1=w'_2=0$, $w'_3= w_3+\epsilon_1 r_3/r_1$
and $w'_4=w_4+\epsilon_2 r_4/r_2$.  Then $R(w')=R(w)$, $Q(w')=Q(w)$,
$A(w')<A(w)$ and $B(w')<B(w)$.  This is a contradiction since
neither constraint~$A$ nor constraint~$B$ is tight.
\end{proof}

For the balance of the argument, since we will be concerned exclusively
with $\ple$'s on the balancing line (where $x=y$), we abbreviate $r(x,x)$ to
$r(x)$.  Similarly, $q(x,x)$ will be written as $q(x)$, and
the $\ple$ $M(x,x)$ will now be just $M(x)$.  Also, we will only be
concerned with constraints $R$, $Q$ and~$B$. 

Now for our fourth tie-breaker.

\begin{enumerate}[wide=0pt]
\item[{$\bTB\bf{(4)}$}.] Minimize the number $t_4$ of $\ple$'s in $\ple(P)$
which are assigned positive weight by $w$.
\end{enumerate} 

In view of the statement of our main theorem, and our outline
for the proof, it is clear that ultimately we will show that $t_4\le2$,

The following proposition consists of some relatively straightforward
calculus exercises.

\begin{proposition}\label{pro:delta-gamma}\hfill

\begin{enumerate}
\item The function $r(x)/x$, defined on the interval $(0,1]$, is 
concave downwards on the interval $(0,1]$, has a vertical asymptote at $x=0$ and achieves a local
maximum value at $x=\delta$ with $\beta<\delta<x_\bst$.
\item The function $q(x)/x$ is monotonically increasing on $[\beta,1]$.
\end{enumerate}
\end{proposition}

Since $q(x)/x$ and $r(x)/x$ are increasing on the closed interval
$[\beta,\delta]$, the following result is a
straightforward application of Proposition~\ref{pro:4-way} in 
the continuous setting.

\begin{lemma}\label{lem:below-delta}
There is no local weight function $w\in\cgW_0$ which
assigns positive weight to a $\ple$ $M=M(x)$ with
$\beta\le x<\delta$.
\end{lemma}

Next, we identify three balanced $\ple$'s, to be
denoted $M_\bst$, $M_\bal$ and $M_\sat$.  The subscripts
in this notation are abbreviations, respectively, for
best, balanced and saturated.
 
Recall, from Section ~\ref{sec:force}, that $x_\bst$ is the unique point in the interval
$[\beta,1]$ where the function $g(x)=[(d+1)r(x)-x]/[(d+1)r(x)-x^2]$
achieves its maximum value. Also, $x_\bal$ is the unique point
from $[\beta,1]$ where $r(x)=q(x)$.  Straightforward
calculations show $\beta<\delta<x_\bst<x_\bal<1$.
Set $M_{\bst}=M(x_\bst)$ and $M_\bal=M(x_\bal)$.

With $d$ fixed and $n\rightarrow\infty$, the block-regular
$\ple$ $M(\binom{n}{d},n)$ converges to the linear extension
which we now denote as $M(1)$.  We note that $r(1)=1/(d+1)$ and 
$q(1)=d/(d+1)$.  To be consistent with our notation for
$M_\bst$ and $M_\bal$, we set $x_\sat=1$ and $M_\sat=M(1)$.

\begin{lemma}\label{lem:no-far-out}
There is no $w\in\cgW_0$ which assigns positive weight to
a $\ple$ $M=M(x)$ with $x_\bal <x <1$.
\end{lemma}

\begin{proof}
We argue by contradiction and assume some $w\in\cgW_0$ assigns
positive weight $w_0$ to a $\ple$ $M_0=M(x_0)$ with $x_\bal<x_0<1$.
Let $M_1=M_\bal$ and $M_2=M_\sat$.
We will form a local weight function $w'$ from $w$ by removing
all weight from $M_0$ and increasing the weight on $M_1$ and
$M_2$ so that $R(w')=R(w)$, $Q(w')=Q(x)$ and $B(w')<
B(w)$.  To accomplish this task, we must find
positive values $v_1=w'_1-w_1$ and $v_2=w'_2-w_2$ that satisfy the 
following three requirements:
\begin{align*}
w_0 r_0&\le  v_1 r_1+ v_2/(d+1)\\ 
w_0 q_0&\le  v_1 q_1+ v_2 d/(d+1)\\ 
w_0 x_0 &>   v_1 x_1 + v_2.
\end{align*}

We consider the first two requirements as equations.  Noting
that $q_1=r_1$, this results in the following solution:
\begin{align*}
v_1&=w_0(dr_0-q_0)/[(d-1)r_1]\\
v_2&=(d+1)w_0(q_0-r_0)/(d-1).
\end{align*}
From Proposition~\ref{pro:max-q/r}, we know $q_0/r_0<d$ which
shows that $v_1>0$.  Also, since $x_\bal<x_0<1$, we know $q_0>r_0$
which implies $v_2>0$.

Using the fact that $v_2$ can be also be written as
$(d+1)(w_0r_0-v_1r_1)$, the third inequality is equivalent to:
\[
\frac{(d+1)r_1-x_1}{(d+1)r_1-x_1^2} > \frac{(d+1)r_0-x_0}{(d+1)r_0-x_0^2}.
\]
However, we know the function $g(x)$ is decreasing when $x\ge x_\bst$,
so this last inequality holds.  The contradiction completes the proof
of the lemma.
\end{proof}

\begin{lemma}\label{lem:no-total-preserving}
Let $w\in\cgW_0$.  Then there
is no maximal-preserving linear extension
which is assigned positive weight by $w$.
\end{lemma}

\begin{proof}  
Suppose to the contrary that some $w\in\cgW_0$ assigns positive
weight $w_0$ to the maximal-preserving linear extension 
$L_0=[\Min(P)<\Max(P)]$.  Since constraint~$R$ is tight, there
must be some $\ple$ $M_1$ with $w_1>0$ and $r_1>q_1$.  Let
$M_2=M_\sat$. We will then obtain $w'$ from $w$ by setting $w'_0
=w_0-v_0$, $w'_1=w_1-v_1$ and $w'_2=w_2+v_2$. 
We require $0<v_0\le w_0$, $0<v_1\le w_1$ and
$v_0/v_1= dr_1-q_1$.  Clearly, appropriate values of $v_0$ and 
$v_1$ can be chosen.  We then set $v_2=(d+1)v_1r_1$.

With these choices, it follows that $R(w')=R(w)$ and
$Q(w')=Q(w)$. 
 However,
\begin{align*}
B(w')&=B(w)-v_0-v_1x_1+v_2\\
     &=B(w')-v_1(dr_1-q_1)-v_1x_1+(d+1)v_1r_1\\
     &=B(w')-v_1[dr_1-(x_1^2-r_1)-x_1+(d+1)r_1]\\
     &=B(w)-v_1(x_1-x_1^2)\\
     &<B(w).
\end{align*}
Recall that all $\ple$'s with positive weight are balanced, by Lemma~\ref{lem:balanced}, so we also have $A(w')<A(w)$, a contradiction which  completes the proof of the lemma.
\end{proof}

\begin{lemma}\label{lem:x_bst}
Let $w\in\cgW_0$.  
If $w$ assigns positive weight to a $\ple$ $M=M(x)$ with
$\delta\le x<x_\bal$, then $x=x_\bst$.
\end{lemma}

\begin{proof}
Suppose to the contrary that $w$ assigns positive weight
to a $\ple$ $M_0=M(x)$ with $\delta\le x<x_\bal$ and $x\neq
x_\bst$.  Since $r_0>q_0$, $w$ must also assign positive
weight to $M_2=M_\sat$.  We form $w'$ by making three changes.
First, we set $w'_0=w_0-v_0$, and we will
require that $0< v_0\le w_0$.   Also, there may be an
additional restriction on the size of $v_0$, a detail
that will soon be clear.  Second, we will take $M_1=M_\bst$ and
set $w'_1=w_1+v_1$ with $v_1>0$.  Third, we will take
$M_2=M_\sat$ and set $w'_2=w_2+v_2$, however we place 
no restrictions on the sign of $v_2$, except that we need
$w'_2\ge0$, so if $v_2$ is negative, we need
$|v_2|\le w_2$.  This will be handled if necessary by
restricting the size of $v_0$.

Our choices will be made so that $R(w')=R(w)$, $Q(w')=Q(w)$ and
$B(w')<B(w)$.  This requires:
\begin{align}\label{eqn:required}
v_0 r_0 &=  v_1 r_1 + v_2/(d+1),\\ 
v_0 q_0 &=  v_1 q_1 + v_2 d/(d+1),\\
v_0 x_0 &>  v_1 x_1 + v_2.
\end{align}

Solving the first two equations, we obtain
$v_1=(dr_0-q_0)/(dr_1-q_1)$ so that $v_1>0$.  Note that
this equation can be rewritten as $v_1=[(d+1)r_0-x_0^2]/[(d+1)r_1-x_1^2]$.
We also obtain:
\begin{align*}
v_2&=(d+1)(v_0r_0-v_1r_1).\\
\end{align*}
Accordingly, when $v_2<0$, we can maintain $|v_2|\le w_2$ by scaling the system~\ref{eqn:required} to keep $v_0$ sufficiently small.

After substituting the specified values, the third constraint $v_0x_0>v_1x_1+v_2$
becomes $g(x_1)>g(x_0)$, which is true since $x_0\neq x_1$. Since $w'$ can be obtained with the desired properties, we have contradicted $w\in \cgW_0$, completing the proof of the lemma.
\end{proof}

With these observations in hand, our linear programming problem
reduces to:

\smallskip
Minimize the quantity $c$ subject to the following constraints:
\begin{align*}
  B(w)&=w(M_\bst)x_\bst+w(M_\bal)x_\bal+w(M_\sat)\le c\\
  R(w)&=w(M_\bst)r(x_\bst)+w(M_\bal)r(x_\bal)+w(M_\sat)/(d+1)\ge1\\
  Q(w)&=w(M_\bst)q(x_\bst)+w(M_\bal)q(x_\bal)+w(M_\sat)d/(d+1)\ge1
\end{align*}

By increasing $w(M_\bal)$, we can decrease $w(M_\bst)$ and $w(M_\sat)$ or vice versa to see that there are optimal solutions with at most two non-zero values
among the variables $w(M_\bst)$, $w(M_\bal)$ and $w(M_\sat)$.
Furthermore, the form of the constraints tell us that the only 
possible solutions are (1)~to put weight only on $M_\bal$, or 
(2)~to put positive weight only on $M_\bst$ and $M_\sat$.
The values of $c_0$ given in the statement of our main theorem
reflect the calculations of $c_0$ which these two options
would yield.  With these remarks, the
proof of our lower bound is complete.

\section{Proof of the Upper Bound}\label{sec:upper-bound}

In this section, we show that the lower bound obtained in the
preceding section is also an upper bound.  Our lower bound
is the minimum of two quantities, so in this section we show
that each of the two options provides an upper bound.  As before, our
treatment is informal.

Recall that $M(x)$ is a block-regular ple with $x=b/\binom{n}{d} = a/n$. In particular, $M_{\bal} = M(x_{\bal})$, $M_{\bst} = M(x_{\bst})$, and $M_{\sat} = M(1)$ where $x_{\bal}$ and $x_{\bst}$ were defined in Section~\ref{sec:force}.

We start with the first option in which all weight is put on
the $\ple$ $M_0=M_\bal$ so that $x_0=x_\bal$.
We have $r_0=r(x_0)=x_0^2/2$. Therefore
$w_0=w(M_0)=1/r_0=2/{x_0^2}$.  Also, we have
$c_0=w_0x_0=x_0/r_0=2/x_0$. 
To construct a fractional local realizer for $P$,
we then distribute the weight $w_0$ evenly among all $\ple$'s
which are isomorphic to $M_0$.  With this rule, we will have:

\begin{enumerate}
\item For a given pair $(S,S')$ of distinct sets in $\Max(P)$,
the total weight assigned to $\ple$'s $M$ with $S>S'$ in
$M$ is $w_0x_0^2/2=1$.
\item For a given pair $(u,u')$ of distinct sets in $\Min(P)$,
the total weight assigned to $\ple$'s $M$ with $u>u'$ in
$M$ is $w_0x_0^2/2=1$.
\item For a pair $(u,S)\in\Min(P)\times\Max(P)$ with
$u\not\in S$, the total weight assigned to $\ple$'s $M$ with
$u>S$ in $M$ is~$w_0r_0=1$.
\item For a pair $(u,S)\in\Min(P)\times\Max(P)$ with
$u\not\in S$, the total weight assigned to $\ple$'s $M$ with
$u<S$ in $M$ is~$w_0(x_0^2-r_0)=w_0r_0=1$.
\end{enumerate}

We are left to consider pairs $(u,S)\in\Min(P)\times\Max(P)$ with
$u\in S$.  For the $\ple$ $M_0$, the fraction of $\Min(P)$ that
is in the top box is $1-x_0^{1/d}$, so the total weight assigned
to $\ple$'s $M$ with $u<S$ in $M$ is $w_0 x_0[x_0-(1-x_0^{1/d})]$.
We need this quantity to be at least~$1$, but this is equivalent
to requiring that $r_0=x_0^2/2\ge x_0(1-x_0^{1/d})$, and a
straightforward computation shows that this inequality is valid.

For the second option, we take $M_1=M_\bst$, $x_1=x_\bst$,
$M_2=M_\sat$ and $x_2=1$.  Now we have:
\begin{align}\label{eqn:two-parts}
\begin{split}
w_1r_1+w_2/(d+1)&=1,\\
w_1q_1+w_2d/(d+1)&=1,\\
w_1x_1+w_2&=c_0.
\end{split}
\end{align}

We take all $\ple$'s of $P$ which are isomorphic to
$M_1$ and distribute the weight $w_1$ evenly among them.
Similarly, we take all $\ple$'s that are isomorphic
to $M_2$ and distribute the weight  $w_2$ evenly among them.
We then have:

\begin{enumerate}
\item For a given pair $(S,S')$ of distinct sets in $\Max(P)$,
the total weight assigned to $\ple$'s $M$ with $S>S'$ in
$M$ is $w_1x_1^2/2+w_2/2$.  However, if we add the first
two of the equations in~\eqref{eqn:two-parts} and divide
by~$2$, we obtain:
\[
w_1(r_1+q_1)/2+w_2/2= w_1x_1^2/2+w_2/2=1.
\]
\item For a given pair $(u,u)$ of distinct sets in $\Min(P)$,
the total weight assigned to $\ple$'s $M$ with $u>u'$ in
$M$ is $w_1x_1^2/2+w_2/2$.  Again, this sum is~$1$.
\item For a pair $(u,S)\in\Min(P)\times\Max(P)$ with
$u\not\in S$, the total weight assigned to $\ple$'s $M$ with
$u>S$ in $M$ is~$w_1r_1+w_2/(d+1)=1$.
\item For a pair $(u,S)\in\Min(P)\times\Max(P)$ with
$u\not\in S$, the total weight assigned to $\ple$'s $M$ with
$u<S$ in $M$ is~$w_1(x_1^2-r_1)+w_2d/(d+1)$.  This reduces
to the second equation of \eqref{eqn:two-parts}.
\end{enumerate}

Again, we are left to consider pairs $(u,S)\in\Min(P)\times\Max(P)$ with
$u\in S$.  The desired inequality is:
\[
w_1x_1[x_1-(1-x_1^{1/d})]+w_2\ge 1.
\]

To obtain this inequality, we first observe that the inequality
$d+1-dw_1r_1\ge1$ is equivalent to $1\ge w_1r_1$, but this inequality
holds strictly in view of the first equation of~\eqref{eqn:two-parts}. The first two equations in \eqref{eqn:two-parts} imply $w_2=2-w_1x_1^2$.
It follows that:
\begin{align*}
w_1x_1[x_1-(1-x_1^{1/d})]+w_2&= w_1x_1[x_1-(1-x_1^{1/d})]+2-w_1x_1^2\\
                             &= -w_1x_1(1-x_1^{1/d}) + 2\\
                             &\ge -w_1r_1+2\\
                             &\ge 1.
\end{align*}
With these observations, the proof of the upper bound is complete.

\section{Posets of Bounded Degree}\label{sec:maximum-degree}

Let $d$ be a non-negative integer.
One of the long-standing problems in dimension theory is to
determine as accurately as possible the quantity $\MD(d)$ which we
define to be the maximum dimension among all posets in which each
element is comparable with at most~$d$ other elements.  Trivially,
$\MD(0)=\MD(1)=2$.  With a little effort, one can show that
$\MD(2)=3$.  On the other hand, it is not at all clear that
$\MD(d)$ is well-defined when $d\ge3$.  

However, in 1983,
 R\"{o}dl and Trotter\cite{bib:Trot-Book} proved that $\MD(d)$ is well-defined
and satisfies $\MD(d)\le 2d^2+2$.  Subsequently, F\"{u}redi
and Kahn~\cite{bib:FurKah} proved in 1986 that $\MD(d)=O(d \log^2 d)$,
and it is historically significant that their proof was
an early application of the Lov\'asz local lemma. 

From below, the standard examples show $\MD(d)\ge d+1$.  However,
in 1991, Erd\H{o}s, Kierstead and Trotter~\cite{bib:ErKiTr} improved 
this to $\MD(d)=\Omega(d\log d)$. From 1991 to 2018, there were
no improvements in either bound, but after a gap of nearly 30 years,
Scott and Wood~\cite{bib:ScoWoo} have recently made the
following substantive improvement in the upper bound.

\begin{theorem}\label{thm:SW-technical}
If $d\ge 10^4$ and $d\rightarrow\infty$, then
\[
\MD(d)\le (2e^3+o(1)) (d\log d)\bigl(e^{2\sqrt{\log \log d}}\bigr)\log\log d.
\]
\end{theorem}
As a consequence, we know $\MD(d)\le d\log^{1+o(1)} d$.

Of course, the analogous problem can be considered for any
of the variants of dimension, and in some cases, the following
phenomenon is then observed:  There is a bound on the parameter
for the class of height~$2$ posets in which each maximal element
is comparable with at most $d$ minimal elements \emph{independent}
of how many maximal elements are comparable with a single
minimal element.  As our results show, fractional
local dimension is one such parameter. Fractional dimension and
Boolean dimension are two others.  Curiously, local dimension is not, 
since $\ldim(1,2;n)$ goes to infinity with $n$.  

For general posets, one can ask for the maximum value $\MFLD(d)$ of
the fractional local dimension of a poset in which each element
is comparable with at most $d$ others.  Taking the split of such
a poset, the fractional local dimension cannot increase more 
than~$1$, so we have $\FLD(d)+2$ as an upper bound.  However, we
have no feeling as to whether this is (essentially) best possible.

\section{Closing Comments}\label{sec:close}

One natural open problem for this new parameter is to 
determine the maximum fractional local dimension among all posets on
$n$ points.   The answer to the analogous question for fractional
dimension is $\lfloor n/2\rfloor$ when $n\ge4$, and the standard examples
show this inequality is tight. On the other hand, it is shown 
in~\cite{bib:KMMSSUW} that the maximum local dimension of a poset 
on $n$ points is $\Theta(n/\log n)$, and we suspect that this is
also the answer for fractional local dimension.

Other open questions revolve around the interplay between
fractional local dimension and width.  For any poset with width $w\geq 3$, the dimension is at most $w$~\cite{bib:Dilworth}. This inequality is tight for fractional dimension precisely when $P$ contains the standard example $S_w$ ~\cite{bib:FelTro}. However, we don't know if this inequality is tight for local dimension as standard examples have local dimension 3~\cite{bib:BPSTT}. This leaves open the analogous question: What is the maximum value of the fractional local dimension of a poset with width $w$?


\begin{thebibliography}{99}
  
\bibitem{bib:BPSTT} 
F.\ Barrera-Cruz, T.\ Prag, H.\ C.\ Smith, L.\ Taylor and W.\ T.\ Trotter, 
Comparing Dushnik-Miller dimension, Boolean dimension and local dimension, 
\textit{Order} \textbf{37} (2020), 243--269.

\bibitem{bib:BiHaPo} 
C.\ Bir\'{o}, P.\ Hamburger and A.\ P\'{o}r, 
The proof of the removable pair conjecture for fractional dimension, 
\textit{Electron. J. Comb.} \textbf{21} (2014), \#P1.63.

\bibitem{bib:BHPT} 
C.\ Bir\'{o}, P.\ Hamburger, A.\ P\'{o}r and W.\ T.\ Trotter,
Forcing posets with large dimension to contain large standard examples,
\textit{Graphs and Combin.}, \textbf{32} (2016), 861--880.

\bibitem{bib:BoGrTr} 
B.\ Bosek, J.\ Grytczuk and W.\ T.\ Trotter, 
Local dimension is unbounded for planar posets, submitted.  
Available on the arXiv at 1712.06099.

\bibitem{bib:BriSch} 
G.\ R.\ Brightwell and E.\ R.\ Scheinerman, 
On the fractional dimension of partial orders, 
\textit{Order} \textbf{9} (1992), 139--158.

\bibitem{bib:Dilworth}
R. P. Dilworth, 
A decomposition theorem for partially ordered sets, 
\textit{Ann. Math. (2),}~\textbf{41} (1950), 161-166.

\bibitem{bib:Dush} 
B.\ Dushnik, 
Concerning a certain set of arrangements,
\textit{Proc. Amer. Math. Soc.}~\textbf{1} (1950), 788--796.

\bibitem{bib:DusMil} 
B.\ Dushnik and E.\ W.\ Miller, 
Partially ordered sets,
\textit{Amer. J. Math.}~\textbf{63} (1941), 600--610.

\bibitem{bib:ErKiTr}
P.\ Erd\"{o}s, H.\ Kierstead and W.\ T.\ Trotter,
The dimension of random ordered sets,
\textit{Random Structures and Algorithms} \textbf{2} (1991), 253--275.

  
\bibitem{bib:FeMeMi} 
S.\ Felsner, T.\ M\'{e}sz\'{a}ros and P.\ Micek, 
Boolean dimension and tree-width,
\textit{Combinatorica} (2020), 139--158. DOI: s00493-020-4000-9

\bibitem{bib:FelTro} 
S.\ Felsner and W.\ T.\ Trotter, 
On the fractional dimension of partially ordered sets, 
\textit{Discrete Math.} \textbf{136} (1994), 101--117.

\bibitem{bib:FurKah}
Z.\ F\"{u}redi and J.\ Kahn,
On the dimension of ordered sets of bounded degree,
\emph{Order} \textbf{3} (1988), 15--20.

\bibitem{bib:GaNeTa} 
G.\ Gambosi, J.\ Ne\v{s}et\v{r}il and M.\ Talamo, 
On locally presented posets, 
\textit{Theoretical Computer Science} \textbf{70} (1990), 251--260.

\bibitem{bib:Hira} 
T.\ Hiraguchi, 
On the dimension of orders, 
\textit{Sci. Rep.  Kanazawa Univ.},~\textbf{4} (1955), 1--20.
  
\bibitem{bib:HosMor} 
S. Ho\c{s}ten and W. D. Morris, 
The dimension of the complete graph, 
\textit{Discrete Math.} \textbf{201} (1998), 133--139.

\bibitem{bib:JMMTWW}
G.\ Joret, P.\ Micek, K.\ Milans, W.\ T.\ Trotter, B.\ Walczak and R.\ Wang,
Tree-width and dimension,
\textit{Combinatorica} \textbf{36} (2016), 431--450.

\bibitem{bib:Kier} 
H.\ A.\ Kierstead, 
The order dimension of the $1$-sets versus the $k$-sets, 
\textit{J. Comb. Theory Ser. A}~\textbf{73} (1996), (219--228.

\bibitem{bib:KMMSSUW} 
Kim, Martin, Masa\v{r}\'{i}k, Shull, Smith, Uzzell and Wang,
On difference graphs and the local dimension of posets, 
\textit{European J. Combin.} \textbf{86} (2020), 103074.
  
\bibitem{bib:KleMar} 
D. J. Kleitman and G. Markovsky, 
On Dedekind's problem: The number of isotone boolean functions, II, 
\textit{Trans. Amer. Math.  Soc.}~\textbf{213} (1975), 373--390.

\bibitem{bib:MeMiTr} 
T.\ M\'{e}sz\'{a}ros, P.\ Micek and W.\ T. Trotter, 
Boolean Dimension, Components and Blocks, 
\textit{Order} \textbf{37} (2020), 287-298.

\bibitem{bib:NesPud}
J.\ Ne\v{s}et\v{r}il and P.\ Pudl\'ak, \emph{A note on {B}oolean
dimension of posets}, Irregularities of partitions ({F}ert\H od, 1986),
Algorithms Combin. Study Res. Texts, vol.~8, Springer, Berlin, 1989,
pp.~137--140.

\bibitem{bib:ScoWoo}
A.\ Scott and D.\ Wood,
Better bounds for poset dimension and boxicity,
\textit{Trans. Amer. Math. Soc.} \textbf{373} (2020), 2157-2172.

\bibitem{bib:Spen} 
J.\ Spencer, 
Minimal scrambling sets of simple orders,
\textit{Acta. Math. Hungar.} \textbf{22} (1972), 349--353.

\bibitem{bib:Trot-Book} 
W.\ T.\ Trotter, 
\textit{Combinatorics and Partially Ordered Sets: Dimension Theory}, 
The Johns Hopkins University Press, Baltimore, MD, 1992.
  
\bibitem{bib:TroWal} 
W.\ T.\ Trotter and B.\ Walczak, 
Boolean dimension and local dimension, 
\textit{Electron. Notes in Discrete Math.} \textbf{61} (2017), 1047-1053.
  
\bibitem{bib:TroWan}
W.\ T.\ Trotter and R.\ Wang,
Dimension and matchings in comparability and incomparability graphs,
\textit{Order}~\textbf{33} (2016), 101--119 (with R. Wang).

\bibitem{bib:TroWin} 
W.\ T.\ Trotter and P.\ M.\ Winkler, 
Ramsey theory and sequences of random variables, 
\textit{Probability, Combinatorics and Computing} \textbf{7} (1998), 221--238.

\bibitem{bib:Ueck} 
T. Ueckerdt, 
personal communication.  
\end{thebibliography}
\end{document}